\documentclass[10pt]{amsart}

\usepackage[T1]{fontenc}
\usepackage[utf8]{inputenc}
\usepackage[a4paper,margin=1.15in]{geometry}
\usepackage{amsmath,amssymb,amsthm,mathtools}
\usepackage{enumitem}
\usepackage{microtype}
\usepackage{hyperref}
\usepackage[nameinlink,noabbrev]{cleveref}

\hypersetup{
  colorlinks=true,
  linkcolor=blue,
  citecolor=blue,
  urlcolor=blue
}

\newtheorem{theorem}[equation]{Theorem}
\newtheorem*{thm*}{Theorem}

\newtheorem{corollary}[equation]{Corollary}
\newtheorem{lemma}[equation]{Lemma}
\newtheorem{proposition}[equation]{Proposition}

\theoremstyle{definition}

\newtheorem{definition}[equation]{Definition}
\newtheorem{defn-lem}[equation]{Definition-Lemma}
\newtheorem{remark}[equation]{Remark}
\newtheorem*{rem*}{Remark}

\numberwithin{equation}{section}

\newcommand{\PA}{\mathrm{PA}}
\newcommand{\ZF}{\mathrm{ZF}}
\newcommand{\EA}{\mathrm{EA}}

\newcommand{\Prov}{\mathrm{Prov}}
\newcommand{\Proof}{\mathrm{Proof}}

\newcommand{\Con}{\mathrm{Con}}
\newcommand{\ConS}{\mathrm{Con}^S}
\newcommand{\RFN}{\mathrm{RFN}}

\newcommand{\ul}[1]{\ulcorner #1\urcorner}
\newcommand{\num}[1]{\overline{#1}}
\newcommand{\N}{\mathbb{N}}
\newcommand{\GLP}{\mathrm{GLP}}

\newcommand{\Tr}{\mathrm{Tr}}

\title[Uniformity of Consistency: Ein M\"archen]{Uniformity of Consistency in Arithmetic and G\"odel's Second Incompleteness Theorem: Ein M\"archen}
\author{Harald Grobner}
\address{Harald Grobner, Faculty of Mathematics, University of Vienna, Oskar-Morgenstern-Platz 1, A-1090 Vienna, Austria}
\email{harald.grobner@univie.ac.at}


\date{\today}

\begin{document}

\begin{abstract}
In much discussed work Artemov has recently argued that, for $\PA$, the consistency schema admits a form of uniform verification via selector-proofs, despite the unprovability of the corresponding uniform consistency sentence $\Con(\PA)$. In this note, we show that this phenomenon extends to all sufficiently strong, uniformly reflexive arithmetizable theories, including $\ZF$ and many of their extensions: For such theories $T$, there exists a primitive recursive selector which, given a derivation code $d$, extracts a finite fragment $T_d\subseteq T$ containing the non-logical axioms occurring in $d$, uses a reflexivity proof of $\Con(T_d)$, and produces a $T$-proof that $d$ is not a derivation of $0=1$. As a dictum, one obtains a $T$-verification of the consistency of $T$ in a uniform way, despite the fact that it cannot be internalized as the single universal consistency sentence prohibited by G\"odel's Second Incompleteness Theorem. We further analyze this latter discrepancy and locate selector-proofs within the broader framework of provability and reflection.


\end{abstract}

\maketitle

\section*{Introduction}
\noindent G\"odel's Second Incompleteness Theorem (2.UVS\footnote{The unconventional abbreviation stands for {\it Zweiter (=2.) Unvollst\"andigkeitssatz}.}), cf.\ \cite{Godel1931}, is commonly taken to establish a definitive limitation on formal theories: No sufficiently strong, recursively axiomatized, consistent theory \(T\), based on a system of axioms $\mathcal L$, can prove its own consistency, provided that consistency is expressed in the standard arithmetical form
\[
  \Con(T) \equiv \forall x\, \neg \Proof_T(x,\ul{0=1}).
\]
This formulation has played a central role in the interpretation of Hilbert's program and in subsequent debates about the epistemological significance of consistency proofs, cf.\  \cite{Resnik1974,Auerbach1978,Auerbach1989,Steiner1991}.  The standard conclusion, drawn from G\"odel's 2.UVS, is that a theory capable of formalizing a sufficiently rich fragment of arithmetic cannot, by its own resources, certify the absence of a derivation of contradiction from its own axioms.\\\\
There is, however, a complementary perspective on this limitation, arising from the analysis of {\it serial properties} and their formalization, leading to a form of uniform provability which is not captured by a single universal sentence: For each standard natural number \(n\), let
\[
  \Con_n(T) \equiv \neg \Proof_T(\num n,\ul{0=1}),
\]
and consider the associated {\it consistency schema}
\[
  \ConS(T) := \{\Con_n(T) : n\in\N\}.
\]
It was shown in \cite{Artemov2019,Artemov2025JLC} that, for \(\mathcal L=\PA\), this schema is provable in the sense of so-called {\it selector-proofs}: That is, there exists (i) a ``genuine mathematical'' proof that no derivation of $T$ ends with $0=1$ and (ii) a G\"odelian formalization of this proof inside the theory. More precisely, for $\PA$, part (i) is supplied by partial truth predicates: Given a derivation $d$, one chooses a suitable partial truth predicate ${\rm Tr}_n$ and proves, using the Tarski conditions, that all formulas in $d$ are $\Tr_n$-true; since $0=1$ is not $\Tr_n$-true, $d$ cannot be a derivation of $0=1$. Together with a G\"odelian formalization of this argument, which is the contents of (ii), this finally yields the aforementioned ``selector-proof'' of the corresponding consistency schema. Thus, while the universal sentence \(\Con(\PA)\) is unprovable, the corresponding series $ \ConS(\PA)$ still admits a strong form of uniform witnessing in $\PA$.\\\\
As the first main goal of this short note\footnote{A note, which admittedly bears a title that can rightly be considered stolen from R. P. Langlands, or at least contains a clear allusion to a work written by him and highly valued by the author of these lines, see \cite{LanglandsCorvallis}.}, we show that this phenomenon is not specific to \(\PA\), but extends in a natural way to all sufficiently strong arithmetizable theories.\\\\
This is done a a two-step analysis: First we record a well-known general scheme-level construction, implicit in the work of many people, for instance Pudl\'ak,
Feferman, and Verbrugge--Visser \cite{Pudlak1986,Fef62,Ver-Vis94}: For any recursively axiomatized theory $T$, whose syntax is representable in a weak arithmetic base, there exists a primitive recursive procedure producing proofs of all instances $\Con_n(T)$. See Prop.\ \ref{prop:selector-con} and Rem.\ \ref{rem:Pudlak} for all details. Second, if \(T\) is uniformly reflexive, in the sense that there is a primitive recursive procedure which, given a finite fragment \(F\subseteq T\), produces a \(T\)-proof of the G\"odelian consistency statement $\Con(F)$, then a derivation code $d$ can be handled as follows: Extract the finite fragment $T_d\subseteq T$ actually used in $d$, use the $T$-proof of $\Con(T_d)$, and infer that $d$ is not a proof of $0=1$. This gives the additional consistency-checking step missing from the scheme-level construction: The selector no longer merely produces proofs of the schema instances, but does so by certifying the finite fragment actually used by the given derivation. This extends the proof-producing mechanism of Artemov from \(\PA\) to the above uniformly reflexive theories, including \(\PA\) and, under the usual coding of syntax, \(\ZF\) and many of their extensions. We refer to our Thm.\ \ref{thm:-selector}, Cor.\ \ref{cor:pa-extensions-} and Cor.\ \ref{cor:zf-selector}.\\\\
The second main goal of this note is to locate selector proofs within
the broader framework of provability and reflection, and thereby to
draw conclusions about their possible significance for Hilbert's
program, more specifically, for its attempt to clarify in what sense consistency
can, and in what sense it cannot, be proved within the theory itself. Indeed, in that spirit, the above results suggest a refinement of the usual (rather apodictic) interpretation of G\"odel's 2.UVS: The relevant boundary is not simply between provability and unprovability of consistency statements within a given theory, but between {\it two distinct forms of uniformity}: On the one hand, there is a form of proof-producing uniformity, embodied in a single effective procedure, which gives proofs of all instances $\Con_n(T)$ while verifying a given derivation code \(d\) by extracting the finite fragment \(T_d\subseteq T\) actually used in \(d\) and appealing to a \(T\)-proof of \(\Con(T_d)\) (i.e., it provides proofs in a way, which is in this sense uniform). On the other hand, there is the ordinary uniform consistency assertion, expressed by the single universal arithmetical sentence $\Con(T)$ and prohibited by G\"odel's 2.UVS.\\\\
Selector-proofs show, by circumnavigating non-standard natural numbers, which could give rise to proof-codes, which produces ``G\"odel monsters'', that the former can exist even when the latter is provably unattainable. In this sense, they isolate a genuine
intermediate level of proof: stronger than mere and instantwise
verification of the consistency schema, yet strictly weaker than the uniform consistency statement. In
particular, the limitation identified by G\"odel's 2.UVS {\it is not a
limitation on the existence of uniform computational witnesses}, but on
the passage from such witnesses to the ordinary universal consistency
sentence. From this perspective, the significance of the schema-based
constructions underlying selector-proofs clearly lies not in challenging G\"odel's 2.UVS, but in
revealing a form of uniform computational content that does not
collapse into $\Con(T)$. We leave it to the reader to compare this thought to an approach carried out in \cite{Artemov2025}, where Artemov argues that consistency as a schema is provably in $\PA$ equivalent to the standard combinatorial definition of consistency, whereas consistency as a formula $\Con(\PA)$ is not. 

\bigskip
\small
\noindent  {\it Acknowledgments:}
This short note is an outgrowth of the author's lecture-course "Philosophie und Mathematik", held at the University of Vienna in spring 2026 together with Georg Schiemer. We are very grateful to Sergei Artemov for several valuable comments on a preliminary version of this paper and a fruitful discussion about and a presentation of his work and its context. We are also grateful to Emil Je\v r\'abek for hinting us to Pudl\'ak's paper \cite{Pudlak1986}.
\normalsize

\section{Standing assumptions and minor recaps}\label{sec:arith}
\subsection{First-order theories}\label{sect:1}
Unless otherwise stated, in this note, \(T\) denotes a recursively axiomatized theory, formulated in first-order predicate logic (out of a system $\mathcal L$ of axioms, say), extending a weak arithmetic base sufficient for elementary syntactic coding. By the latter we shall (at least) mean elementary arithmetic $\EA$ in the sense of
$I\Delta_0 + \mathrm{Exp}$, i.e.\ Robinson arithmetic together with
induction restricted to $\Delta_0$-formulas and the totality of
exponentiation. However, for stronger metamathematical arguments, in particular those involving
derivability conditions and reflection principles, we work over
$I\Sigma_1$, i.e., arithmetic with induction restricted to
$\Sigma_1$-formulas and we shall indicate this in the statement of our results.

\subsection{Basics on provability and consistency}\label{sect:12}
Given a system of axioms $\mathcal L$ with theory $T$ as in Sect.\ \ref{sect:1}, we choose and fix a standard G\"odel numbering and a standard proof predicate
\[
  \Proof_T(p,x),
\]
where \(\Proof_T(p,x)\) expresses that \(p\) codes a \(T\)-derivation of the formula with G\"odel number \(x\). The very details in its formalization as an expression of $\mathcal L$ are assumed to be fixed once and for all throughout the paper. 

\begin{definition}
The associated {\it provability predicate} is
\[
  \Prov_T(x) \equiv \exists p\,\Proof_T(p,x).
\]
The standard {\it consistency sentence} is
\[
  \Con(T) \equiv \neg\Prov_T(\ul{0=1}),
\]
i.e., $\Con(T) \equiv \forall p\,\neg \Proof_T(p,\ul{0=1}).$
\end{definition}

\begin{lemma}\label{lem:proof-predicate}
For a standard arithmetization of syntax, \(\Proof_T(p,x)\) is elementary, indeed primitive recursive, and hence representable in \(\EA\).  In common codings it may be taken to be \(\Delta_0\) after the usual coding of bounded computations.
\end{lemma}

\begin{proof}
This is the standard representability theorem for primitive recursive relations.  (Verification that a finite sequence is a formal derivation is obtained by checking, for each line, whether it is an axiom or follows from earlier lines by one of finitely many rules.  Since the axiom set $\mathcal L$ of \(T\) is recursive, this verification is primitive recursive relative to the chosen recursive axiom predicate, and in the usual elementary presentations it is representable already in \(\EA\).  See, for example, \cite[Chp.~I]{HP}.)
\end{proof}

\noindent When $T$ extends a sufficiently strong fragment of arithmetic, such as $I\Sigma_1$ (and, for standard codings, already $\EA$ in the usual presentations), the standard Hilbert--Bernays--L\"ob derivability conditions, underlying the usual proof of L\"ob's theorem and G\"odel's 2.UVS, hold for the standard provability predicate $\Prov_T$, assuming a standard coding of syntax, see, e.g., \cite{HilbertBernays1939,Lob1955}, \cite[Ch.~I--II]{HP}, and \cite{Smorynski1985}. For weaker base theories, such as Robinson arithmetic $Q$, these conditions
may fail, in particular condition (D3):

\begin{enumerate}[label=(D\arabic*)] 
  \item If $T\vdash \varphi$, then $T\vdash \Prov_T(\ul{\varphi})$.
  \item $T\vdash \Prov_T(\ul{\varphi\to\psi})
        \to(\Prov_T(\ul{\varphi})\to\Prov_T(\ul{\psi}))$.
  \item $T\vdash \Prov_T(\ul{\varphi})
        \to \Prov_T(\ul{\Prov_T(\ul{\varphi})})$.
\end{enumerate}

\begin{remark}[Jeroslow's trick]
The third derivability condition (D3) is not strictly necessary for the derivation of G\"odel's 2.UVS. Following Jeroslow, cf.\ \cite{Jeroslow1971}, one may replace the standard consistency statement by a slightly modified formula for which the second incompleteness theorem can be established using only (D1) and (D2). Thus, the essential content of the theorem does not depend on the full strength of the Hilbert--Bernays--L\"ob conditions.
\end{remark}

\begin{definition}\label{def:schema}
For each standard \(n\in\N\), define the \(n\)-th {\it consistency instance}
\[
  \Con_n(T) \equiv \neg \Proof_T(\num n,\ul{0=1}).
\]
The associated {\it consistency schema} is the external family
\[
  \ConS(T) := \{\Con_n(T): n\in\N\}.
\]
\end{definition}

\noindent  Here, \(\ConS(T)\) is not a single formula of the object language, but a meta-level family of formulas indexed by the standard natural numbers. It is, however, naturally represented by a primitive recursive arithmetical term that enumerates G\"odel numbers of formulas in  \(\ConS(T)\) and hence may be viewed as a finitary object.

\begin{proposition}\label{prop:pointwise-strong}
For each standard $n\in\N$, the statement
$\Proof_T(\num n,\ul{0=1})$ is decidable by a finite computation
formalizable in $\EA$. In particular, if $T$ is consistent and
extends $\EA$, then
$$T \vdash \Con_n(T).$$
\end{proposition}

\begin{proof}
By Lem.\ \ref{lem:proof-predicate}, the predicate $\Proof_T(p,x)$ is primitive recursive and hence decidable on each fixed standard input. Thus, for each standard $n$, $\EA$ proves either $\Proof_T(\num n,\ul{0=1})$ or its negation. If $T$ is consistent, then no standard $n$ codes a $T$-proof of $0=1$, so the second case holds. Hence $\EA \vdash \Con_n(T)$, and therefore $T \vdash \Con_n(T)$ since $T$ extends $\EA$.
\end{proof}

\noindent Quite obviously, the assertion ``$\forall n\, T\vdash \Con_n(T)$'' is not by itself a meaningful consistency notion as it holds independently of whether \(T\) is consistent. Indeed, we get

\begin{theorem}\label{thm:model-separation}
Let \(T\) be as in Sect.\ \ref{sect:1} and assume in addition that it is a consistent extension of $I\Sigma_1$ with the standard provability predicate. Then, the theory
\[
  T + \neg\Con(T) + \ConS(T)
\]
is consistent. In other words, there is a model \(M\models T\) such that $M\models \Con_n(T)$, for every standard $n\in\N$, but $M\models \neg\Con(T)$, or, equivalently, $M\models \exists x\,\Proof_T(x,\ul{0=1}).$ Any such witness is necessarily non-standard.
\end{theorem}

\begin{proof}
Since \(T\nvdash\Con(T)\) by G\"odel's 2.UVS, the theory \(T+\neg\Con(T)\) is consistent.  Moreover, each standard instance \(\Con_n(T)\) is provable in \(T\) by Prop.\ \ref{prop:pointwise-strong}.  Therefore adjoining any finite subset of \(\{\Con_n(T):n\in\N\}\) to \(T+\neg\Con(T)\) preserves consistency.  By compactness, the full theory is consistent and (hence) has a model \(M\).  In \(M\), all standard instances hold by construction, while \(\neg\Con(T)\) also holds. Therefore, if \(a\in M\) witnesses \(\Proof_T(a,\ul{0=1})\), \(a\) cannot be the interpretation of any standard numeral \(\num n\).
\end{proof}

\noindent This theorem gives the semantic reason why the schema does not collapse to the universal sentence: The schema says that each standard numeral fails to be a contradiction proof.  The universal sentence says that every element of every model fails to be such a proof code. Non-standard models separate these requirements: Whereas the consistency schema controls standard proof codes only, the universal sentence also controls non-standard objects in non-standard models.\\\\
The relevant issue is therefore not mere pointwise provability of the instances, but the
proof-producing mechanism by which these instances are obtained. A purely
scheme-level selector (cf.\ Def.\ \ref{def:selector} below) may obtain the desired instance by explosion in
the exceptional case, in which the input really codes a contradiction
proof. By contrast, a proper {\it selector-proof} (cf.\ Def.\ \ref{def:selector-proof-consistency}), to be considered below, will verify an
alleged derivation code \(d\) by extracting the finite fragment
\(T_d\subseteq T\) actually used in \(d\) and appealing to a
\(T\)-proof of \(\Con(T_d)\).

\section{Selector based proofs of consistency}
\subsection{Weak syntactic bases and scheme-level proof production}\label{sec:selectors}

The previous section treated \(\ConS(T)\) as an external family.
There is a first, purely formal way of making this family uniform: One may ask for a single primitive recursive operation which, on input \(n\), returns a \(T\)-proof of the instance \(\Con_n(T)\). This is an important proof-producing construction, but it is not yet what we shall call a {\it selector-proof of consistency}. The reason is that such a construction may produce the desired instance proof by purely formal means, which is problematic from the (philosophical) point of view of trying to give a proper verification of consistency, if the input in fact codes a contradiction proof. In order to keep this distinction also technically explicit, we first record the following weaker scheme-level notion:

\begin{definition}[Scheme-level selector]\label{def:selector}
Let \(F(x)\) be a primitive recursive series of formulas. We say that \(T\) admits a \emph{scheme-level selector} for the series \(F\), if there exists a primitive recursive function \(s_T\) such that
\[
 T\vdash \forall x\, \Proof_T(s(x),\ul{F(\dot x)}).
\]
\end{definition}

This is a deliberately schematic definition. As just mentioned above, the terminology ``scheme-level'' is meant to indicate that this notion only concerns the production of formal proofs of the individual instances, but it does not by itself provide a verification that the input is not a contradiction proof. In the following arguments, nothing will depend on the particular coding.

\begin{proposition}[Scheme-level construction]\label{prop:selector-con}
Let \(T\) be a recursively axiomatized first-order theory extending \(\EA\), cf.\ Sect.\ \ref{sect:1}, with the standard proof predicate \(\Proof_T\). Then the consistency schema \(\ConS(T)\) admits a scheme-level selector. More precisely, there exists a primitive recursive function \(s_T\) such that
\[
 T\vdash \forall x\, \Proof_T\bigl(s_T(x),\ul{\neg \Proof_T(\dot x,\ul{0=1})}\bigr).
\]
\end{proposition}

\begin{proof}
Since \(\Proof_T(x,\ul{0=1})\) is primitive recursive, there exists a primitive recursive procedure deciding, on input \(x\), whether \(x\) codes a \(T\)-proof of \(0=1\). We define a primitive recursive function \(s_T(x)\) by cases.\\\\
\emph{Case 1:} If \(x\) does not code a \(T\)-proof of \(0=1\), then the negative computation verifying \(\neg \Proof_T(x,\ul{0=1})\) can be converted into a formal \(T\)-derivation of
\[
\neg \Proof_T(\num x,\ul{0=1}).
\]
Let \(s_T(x)\) be the code of this derivation.\\\\
\emph{Case 2:} If \(x\) does code a \(T\)-proof of \(0=1\), then from this proof one obtains a \(T\)-derivation of any formula, in particular of
\[
\neg \Proof_T(\num x,\ul{0=1}).
\]
More explicitly, there is a fixed primitive recursive procedure which, given \(x\), produces a \(T\)-derivation of
\[
0=1 \to \neg \Proof_T(\num x,\ul{0=1}).
\]
Let \(s_T(x)\) be the code of the derivation obtained by appending this derivation to the proof coded by \(x\).\\\\
The construction of \(s_T\) is primitive recursive, since proof verification, substitution of numerals, and concatenation of proof codes are primitive recursive operations. Moreover, the correctness of this construction is formalizable already in \(\EA\), and hence in \(T\). Therefore,
\[
 T\vdash \forall x\, \Proof_T\bigl(s_T(x),\ul{\neg \Proof_T(\dot x,\ul{0=1})}\bigr),
\]
as claimed.
\end{proof}

\begin{remark}[Pudl\'ak's finitistic consistency statements]\label{rem:Pudlak}
Prop.\ \ref{prop:selector-con}  gives a self-contained scheme-level selector. Its proof-theoretic core, however, is closely related to, and in the cases of interest translatable from Pudl\'ak's work on finitistic consistency statements \cite{Pudlak1986}, Thm.\ 5.5. Pudl\'ak considers, for an axiomatization \(A\), formulas of the form \(\Con_A(m)\), expressing that there is no \(A\)-proof of contradiction of length at most \(m\). Thus \(\Con_A(m)\) is a bounded-length consistency statement, whereas our instance
\[
 \Con_n(T) \equiv \neg \Proof_T(\num n,\ul{0=1})
\]
is a code-wise statement about one particular alleged proof code.\\\\
For sequential theories and for suitable presentations of their axiomatizations, Pudl\'ak obtains polynomial upper bounds on the lengths of proofs of such finitistic consistency statements. In particular, rightly interpreted, his results yield polynomially bounded proofs of the corresponding bounded consistency statements of \(\PA\) and \(\ZF\), see \cite{Pudlak1986}, p.\ 192. It is therefore implicit that, for theories covered by Pudl\'ak's hypotheses, his bounded-length result implies the existence of primitive recursive proof-producing functions for the individual code-wise instances: given a proof code \(n\), choose a length bound \(m\) large enough to cover the proof coded by \(n\), use the proof of \(\Con_A(m)\), and then infer that \(n\) itself is not a proof of contradiction. In this sense, Pudl\'ak's theorem contains the proof-complexity substance of the scheme-level selector phenomenon considered by Artemov, \cite{Artemov2019,Artemov2025JLC}, and indeed in a quantitative form. The point of Prop.\ \ref{prop:selector-con}  is therefore primarily conceptual, namely to isolate the scheme-level proof-producing phenomenon in an arithmetized form.
\end{remark}

\subsection{Selector-proofs via uniform reflexivity}\label{sec:selector-proof}

\noindent Prop.\ \ref{prop:selector-con} is to be viewed just as a first step, which deliberately only treats a (rather light) scheme-level so far: In Case 2, if the input \(x\) codes a \(T\)-proof of contradiction, the desired target formula is obtained by explosion, in which case the construction produces proofs of the relevant schema instances, but it does not yet provide a consistency check for the alleged derivation code itself. This is exactly the distinction emphasized by Artemov between producing proofs of the consistency scheme and proving the corresponding consistency property, cf.\ \cite[Sec.~5.2]{Artemov2025JLC}: The latter requires an additional argument (called ``contentual'' in the aforementioned reference) showing, for a given derivation, why it cannot end in \(0=1\).\\\\ 
We now isolate a major additional hypothesis, which in fact turns purely scheme-level proof production into a proper mathematical consistency check for each alleged derivation. The key point is that a concrete derivation uses only finitely many non-logical axioms. If the theory can uniformly prove the consistency of each such finite
fragment, then any putative derivation can be ruled out by appealing to the consistency of the finite fragment, which it actually uses.\\\\
Let \(T\) be a recursively axiomatized first-order theory extending \(\EA\), satisfying the standing assumptions of Sect.\ \ref{sect:1}, with the standard proof predicate \(\Proof_T\).  A finite fragment of axioms in $\mathcal L$ is coded by a finite sequence $e=( a_0,\ldots,a_{k-1})$ of G\"odel numbers and we write \(T_e\) for the corresponding finite subtheory of \(T\). The associated proof predicate \(\Proof_{T_e}(p,x)\) is the primitive recursive predicate obtained from \(\Proof_T\) by replacing the axiom predicate for \(T\) by membership in the finite list \(e\). We put
\[
        \Con(T_e)\equiv\;
       \forall p\,\neg \Proof_{T_e}(p,\ulcorner 0=1\urcorner)
\]
and define 


\begin{definition}[Uniform reflexivity]\label{def:uniform-reflexivity}
Let \(T\) be as in Sect.\ \ref{sect:1}. We say that \(T\) is \emph{uniformly reflexive}, if there exists a primitive recursive function \(u_T\) such that for every finite sequence of G\"odel numbers \(e\), coding a finite fragment of axioms in $\mathcal L$, \(u_T(e)\) is the code of a \(T\)-proof of $\Con(T_e)$. 
\end{definition}

\begin{remark}
Equivalently, finite fragments of $\mathcal L$ admit \(T\)-proofs of their G\"odelian consistency in a way which is uniform and primitive recursive in a code of the fragment.
 Uniform reflexivity is therefore a constructive strengthening of ordinary
reflexivity. While ordinary reflexivity asserts only that
\(T\vdash \Con(T_e)\) for each finite fragment \(T_e\subseteq T\), uniform reflexivity additionally requires that such proofs be obtained
primitively recursively from a code of the fragment.
\end{remark}

\begin{definition}[Selector-proofs of consistency]\label{def:selector-proof-consistency}
Let \(T\) be as in Sect.\ \ref{sect:1}. We say that \(T\) admits a \emph{selector-proof of consistency}, if there exists a primitive recursive function \(h_T\) such that, for every derivation code \(d\), \(h_T(d)\) is a \(T\)-proof of $\neg\Proof_T(\num d,\ul{0=1})$, obtained in the following, prescribed way: One extracts the finite fragment \(T_d\subseteq T\), actually used by the derivation coded by \(d\), uses a \(T\)-proof of \(\Con(T_d)\), and derives from this that \(d\) is not a derivation of \(0=1\).
\end{definition}

\begin{theorem}[Selector-proofs for uniformly reflexive theories]\label{thm:-selector}
Let \(T\) be a recursively axiomatized first-order theory extending \(\EA\), satisfying the standing assumptions of Sect.\ \ref{sect:1}, with the standard proof predicate \(\Proof_T\). Assume furthermore that \(T\) is uniformly reflexive. Then, \(T\) admits a selector-proof of consistency. 
\end{theorem}

\begin{proof}
Given a derivation code \(d\), let \(T_d\) be the finite subtheory of \(T\) consisting of the non-logical \(T\)-axioms occurring in the derivation coded by \(d\). The operation \(d\mapsto T_d\) is primitive recursive: One scans the finite proof object coded by \(d\), identifies those lines which are non-logical \(T\)-axioms, and records their G\"odel numbers. Here we adopt the convention that \(T_d\) denotes the empty fragment, if \(d\) is not a well-formed derivation code.\\\\
By uniform reflexivity, there is a primitive recursive procedure producing a \(T\)-proof of
\[
 \Con(T_d)\equiv \forall p\,\neg\Proof_{T_d}(p,\ul{0=1}).
\]
On the other hand, elementary syntactic verification gives, uniformly in \(d\), a \(T\)-proof of
\[
 \Proof_T(\num d,\ul{0=1})
 \rightarrow
 \Proof_{T_d}(\num d,\ul{0=1}).
\]
Indeed, if \(d\) is a \(T\)-derivation of \(0=1\), then all non-logical axioms used in \(d\) belong by construction to \(T_d\). Hence the same finite derivation is already a \(T_d\)-derivation of \(0=1\). Combining the proof of \(\Con(T_d)\) with this syntactic implication yields a \(T\)-proof of $\neg\Proof_T(\num d,\ul{0=1}).$ As all operations involved -- extracting \(T_d\), producing the reflexivity proof of \(\Con(T_d)\), formalizing the implication from \(T\)-proofhood to \(T_d\)-proofhood, and concatenating the resulting proof codes -- are primitive recursive, we may choose these operations to define the desired selector \(h_T\).
\end{proof}

\noindent Unlike the scheme-level selector of Prop.\ \ref{prop:selector-con}, the construction of Thm.\ \ref{thm:-selector} supplies a verification that the given derivation is controlled by a finite fragment, whose consistency is in fact proved in \(T\), whence more than a proof-producing operation for the consistency scheme.\\\\
We now deal with substantial families of examples of theories, to which Thm.\ \ref{thm:-selector} applies.

\begin{proposition}[Uniform reflexivity of PA-extensions]\label{prop:pa-uniform-reflexive}
Let \(T\) be a first-order theory in the language of arithmetic, satisfying the standing assumptions Sect.\ \ref{sect:1}. Assume moreover that \(T\) extends \(\PA\). Then \(T\) is uniformly reflexive.
\end{proposition}

\begin{proof}[Sketch of a proof:]
This is the constructive content of the classical
reflexivity argument due to Mostowski, Pudl\'ak et al. Let \(F\subseteq T\) be a
finite fragment. We argue in \(T\) that no \(F\)-derivation proves
\(0=1\). Given an arbitrary alleged \(F\)-derivation \(d\), choose
effectively from \(d\) and from the finitely many axioms in \(F\) a
partial truth predicate \(\Tr_{m(d,F)}\) of sufficiently high
complexity to apply to all formulas occurring in \(d\). The theory
\(\PA\) proves the relevant Tarski biconditionals
\[
        A \leftrightarrow \Tr_{m(d,F)}(\ul A)
\]
for the formulas \(A\) in this finite syntactic range. Since the
non-logical axioms of \(F\) are axioms of \(T\), the theory \(T\) proves
that these axioms are \(\Tr_{m(d,F)}\)-true. It also verifies, already
over weak arithmetic, that logical axioms are
\(\Tr_{m(d,F)}\)-true and that the inference rules preserve
\(\Tr_{m(d,F)}\)-truth. Hence \(T\) proves that, if \(d\) is an
\(F\)-derivation, then every formula occurring in \(d\) is
\(\Tr_{m(d,F)}\)-true. But \(T\) proves that \(0=1\) is not
\(\Tr_{m(d,F)}\)-true. Therefore \(d\) is not an \(F\)-derivation of
\(0=1\).\\\\
All constructions used here are uniform in a code of \(F\). The
relevant syntactic bounds, the partial truth predicates, the required
Tarski biconditionals, and the formalized induction over derivations
are obtained effectively, indeed primitively recursively, from the code
of \(F\). Thus, there is a primitive recursive function which, given a
code of \(F\), produces a \(T\)-proof of \(\Con(F)\).
\end{proof}

\begin{corollary}[Selector-proof for PA-extensions]\label{cor:pa-extensions-}
Let \(T\) be a first-order theory in the language of arithmetic, satisfying the standing assumptions Sect.\ \ref{sect:1}, and assume that $T$ extends \(\PA\). Then $T$ admits a selector-proof of its own consistency in the sense of Def.~\ref{def:selector-proof-consistency}.
\end{corollary}
\begin{proof}
Combine Prop.~\ref{prop:pa-uniform-reflexive} with Thm.~\ref{thm:-selector}.
\end{proof}

\subsection{An extending example: The case of \texorpdfstring{\(\ZF\)}{ZF}}\label{sec:zf}

As one should expect, the conclusion of Cor.\ \ref{cor:pa-extensions-} applies -- with the usual {\it ante praeparanda} -- to \(\ZF\) set theory: Fix a standard arithmetization of the syntax of first-order set theory and let \(\Proof_{\ZF}(p,x)\) express that \(p\) codes a \(\ZF\)-proof of the formula with G\"odel number \(x\). Since \(\ZF\) interprets weak arithmetic and formalizes primitive recursive syntax, the finite verification argument above applies verbatim.

\begin{definition}
For each standard \(n\), define
\[
 \Con_n(\ZF) \equiv \neg\Proof_{\ZF}(\num n,\ul{0=1}).
\]
Let
\[
 \Con(\ZF) \equiv \forall p\,\neg\Proof_{\ZF}(p,\ul{0=1}).
\]
\end{definition}

\begin{proposition}\label{prop:zf-instances}
For every standard \(n\), \(\ZF\vdash \Con_n(\ZF)\). Moreover, the
proof of each instance may be chosen to be the image of \(n\) under a
primitive recursive scheme-level selector of the kind described in
Prop.\ \ref{prop:selector-con}. If \(\ZF\) is consistent, then each of
these instances is true in the standard model \(\mathbb N\).
\end{proposition}

\begin{proof}
The scheme-level selector assertion follows from
Prop.\ \ref{prop:selector-con}, taking \(\mathcal L=\ZF\) and using the
standard interpretation of elementary arithmetic in \(\ZF\). Explicitly,
for a fixed standard \(n\), verification whether \(n\) is or is not a
\(\ZF\)-proof of \(0=1\) is a finite primitive recursive computation.
If the computation is negative, the corresponding instance
\(\Con_n(\ZF)\) is provable already in the weak syntactic base
interpreted in \(\ZF\). If the computation is positive, then \(n\)
codes an actual \(\ZF\)-proof of contradiction and \(\ZF\) proves the
same instance by explosion. Finally, if \(\ZF\) is consistent, the
positive case never occurs for standard \(n\), and hence all instances
\(\Con_n(\ZF)\) are true in \(\mathbb N\).
\end{proof}

\begin{theorem}\label{thm:zf-godel}
If \(\ZF\) is consistent, then \(\ZF\nvdash \Con(\ZF).\)
\end{theorem}

\begin{proof}
This is G\"odel's 2.UVS applied to \(\ZF\), which interprets enough arithmetic and whose standard proof predicate satisfies the derivability conditions.
\end{proof}

\noindent We summarize the scheme-level part in our
\begin{corollary}\label{cor:zf-artemov-analogue}
For every standard \(n\), \(\ZF\vdash \Con_n(\ZF)\), and these instance
proofs may be selected primitively recursively. If \(\ZF\) is consistent,
then all these instances are true in the standard model \(\mathbb N\).
Nevertheless, assuming \(\ZF\) is consistent, \(\ZF\) does not prove the
corresponding uniform consistency sentence \(\Con(\ZF)\).
\end{corollary}
\noindent Despite this negative result, we get

\begin{proposition}[Uniform reflexivity of ZF]
\label{prop:zf-uniform-reflexive}
Fix the standard arithmetization of the syntax of first-order set
theory and the corresponding proof predicate
\(\Proof_{\ZF}(p,x)\). Then \(\ZF\) is uniformly reflexive in the
sense analogous to Def.\ \ref{def:uniform-reflexivity}.
The same holds for any recursively axiomatized extension \(T\) of
\(\ZF\) in the language of set theory, once finite fragments of \(T\)
are coded by finite lists of \(T\)-axioms in the fixed coding.
\end{proposition}

\begin{proof}[Sketch of a proof:]
The proof is the set-theoretic analogue of
Prop.~\ref{prop:pa-uniform-reflexive}. For a finite fragment
\(F\subseteq\ZF\), a putative \(F\)-derivation \(d\) involves only
finitely many formulas. \(\ZF\) formalizes the relevant finite
syntactic analysis and supplies partial truth, or satisfaction,
predicates of sufficiently high complexity for the finite collection of
formulas occurring in \(d\) and in \(F\). The corresponding finite
Tarski conditions allow \(\ZF\) to verify that the axioms of \(F\) are
satisfied, that logical axioms are satisfied, and that the inference
rules preserve satisfaction. Since \(0=1\) is not satisfied, \(d\)
cannot be an \(F\)-derivation of \(0=1\). The construction of the
partial satisfaction predicate and of the resulting proof is primitive
recursive in a code of \(F\).
\end{proof}

\begin{corollary}[Selector-proof for ZF]\label{cor:zf-selector}
With respect to a fixed standard arithmetization of the syntax of first-order set theory, \(\ZF\) admits a selector-proof of consistency in the sense analogous to Def.~\ref{def:selector-proof-consistency}. More generally, the same holds for every recursively axiomatized extension \(T\) of \(\ZF\) in the language of set theory, with finite
fragments coded in the fixed way.
\end{corollary}

\begin{proof}
Combine Prop.~\ref{prop:zf-uniform-reflexive} with
Thm.~\ref{thm:-selector}.
\end{proof}

\section{Concluding remarks}\label{sec:discussion}

\subsection{Postlegomena} 
Let us finally isolate several conceptual observations, which are suggested by the
preceding analysis.  The usual G\"odelian consistency statement
\[
        \Con(T)\equiv
        \forall p\,\neg\Proof_T(p,\ulcorner 0=1\urcorner)
\]
is one very specific way of formalizing the assertion that \(T\) is
consistent: It compresses all finite verifications into a single
universal arithmetical sentence.  G\"odel's 2.UVS shows that, for a
sufficiently strong consistent theory \(T\), this particular
internalization is not available in \(T\) itself.\\\\
The selector constructions considered in this paper certainly do not challenge
this conclusion.  Rather, they show that it would be misleading to
identify uniformity of consistency with this single universal
formalization.  A selector-proof gives a different kind of uniformity:
not a proof in \(T\) of the sentence \(\Con(T)\), but a uniform
procedure producing, for each given putative proof code \(d\), a
\(T\)-proof of
\[
        \neg\Proof_T(d,\ulcorner 0=1\urcorner).
\]
In its scheme-level form this amounts to a primitive recursive
procedure producing proofs of all instances of the consistency schema
\(\Con^S(T)\).  In its stronger finite-fragment form, the selector
also records the local reason why the alleged derivation fails: It
extracts the finite axiom fragment \(T_d\subseteq T\) actually used in
the derivation and appeals to a \(T\)-proof of \(\Con(T_d)\).\\\\
Thus, the relevant distinction is not simply between uniform and
non-uniform treatments of consistency.  It is rather a distinction
between different ways in which uniformity may be realized.\\\\  
One may
distinguish the following levels:
\begin{enumerate}[label=(Unif\arabic*)]
   \item \emph{Scheme-level selector uniformity}: There is a single
  primitive recursive function producing \(T\)-proofs of all instances
  \(\Con_n(T)\).

  \item \emph{Finite-fragment selector uniformity}: Given a natural
  number \(d\), regarded as a putative derivation code, one uniformly
  extracts the finite fragment \(T_d\) used by \(d\) and obtains, via a
  \(T\)-proof of \(\Con(T_d)\), a \(T\)-proof that \(d\) is not a proof
  of contradiction.

  \item \emph{Full internal uniformity}: \(T\) proves the single
  universal sentence \(\Con(T)\), or an equivalent uniform reflection
  principle.
\end{enumerate}

\noindent G\"odel's 2.UVS excludes the third level, under the usual hypotheses.
It does not exclude the preceding proof-producing forms of uniformity.
The selector therefore marks a precise proof-theoretic gap: local
consistency verifications may be organized by a single effective
procedure without thereby being absorbed into the global sentence
\(\Con(T)\).\\\\
From this perspective, the philosophical conclusion is that there is more than one mathematically meaningful notion
of uniformity of verification of consistency of a given suitably complex theory $T$.  If uniformity is understood as full internal
compression into the single sentence \(\Con(T)\), G\"odel's theorem
gives a negative answer.  If it is understood as uniform
proof-production for the local consistency claims associated with
putative derivations, selector-proofs give a positive answer.  The
boundary revealed by G\"odel's 2.UVS is therefore not the boundary
between uniformity and non-uniformity as such, but the boundary between
proof-producing uniformity and its complete internalization as a
universal consistency assertion.

\subsection{An orthogonal approach -- different provability predicates}
Another way of resisting the standard interpretation of G\"odel's 2.UVS
is to alter the provability predicate itself. This approach is
conceptually distinct from the schema-based strategy discussed above.
Rather than weakening the form of the consistency assertion, one
replaces the standard provability predicate $\Prov_T(x)$ by a predicate
$\Prov_T^*(x)$ which may still be extensionally related to theoremhood,
but which fails to satisfy one or more of the structural conditions
required for the usual derivation of G\"odel's theorem, see, e.g.,
\cite{Detlefsen1979,Detlefsen2001,Visser2001}.\\\\
Recall that under our standing assumptions, cf.\ Sect.\ \ref{sect:1},
the standard provability predicate $\Prov_T(x)$ satisfies the
Hilbert--Bernays--L\"ob derivability conditions, and that these are
essential for the classical proof of G\"odel's 2.UVS. A predicate
$\Prov_T^*(x)$ may therefore evade the theorem, if it violates a
relevant portion of these conditions. In this sense, one may obtain
``consistency statements'' which are provable in $T$, but only at the
cost of departing from the standard structural notion of provability
\`a la Hilbert--Bernays--L\"ob.\\\\
A humble, non-standard example is given by
\[
\Prov_T^*(x) \;\equiv\; \Prov_T(x) \wedge x \neq \ul{0=1}.
\]
Then, $\Prov_T^*(\ul{0=1})$ is trivially refutable already in very weak arithmetic
and hence in $T$. However, this predicate fails to satisfy the
Hilbert--Bernays--L\"ob derivability conditions, as announced.\\\\
More sophisticated constructions, such as Rosser-style provability
predicates, are considerably subtler, see, for instance,
\cite{Smorynski1985,Visser2001}. Their philosophical significance has
been extensively discussed in the literature, most notably by
Detlefsen \cite{Detlefsen1979,Detlefsen1986,Detlefsen2001}. The central
question is whether the Hilbert--Bernays--L\"ob conditions should be
regarded as essential to the very notion of formal provability, or
merely as one convenient formalization among others.\\\\
From the perspective of the present paper, the key point is that this
approach operates at a different level from the schema-based
distinction analysed above. The latter preserves the standard notion
of provability and weakens the formal incarnation of the consistency assertion,
whereas the present strategy modifies the notion of provability
itself. The two approaches are therefore orthogonal: One weakens the
form of the consistency assertion, the other the underlying notion of
provability.

\subsection{A reflection on Reflection}\label{sec:hierarchies}
Let us finally locate the schema/uniformity distinction in relation to reflection hierarchies. 
Reflection principles and their iterated forms play a central role in
the proof-theoretic analysis of arithmetic. They provide a systematic
way of measuring the strength of theories in terms of their internal
soundness properties. The resulting hierarchies of theories, obtained
via iterated reflection or consistency extensions, have been studied
extensively, in particular in the work of Feferman \cite{Fef62} and Beklemishev
\cite{Beklemishev1995,Beklemishev2005,Beklemishev2010}. A closely
related perspective is given by provability algebras and polymodal
provability logics, where such hierarchies are analysed in a modal
framework, allowing for a fine calibration of reflection principles
and their iterations (see also \cite{Visser2001}). In particular, in
Beklemishev's $\GLP$-framework, one considers a family of modal
operators, denoted by $[n]$ (and their duals $\langle n\rangle$), each intended to
formalize a provability predicate of increasing strength. The
corresponding iterated consistency statements, usually denoted $\langle n\rangle \top$,
serve as canonical measures of proof-theoretic strength. For general
background on arithmetized provability and reflection we refer to
\cite{HP}.\\\\
Slightly more precisely, let $T$ be as in Sect.~\ref{sect:1}, and assume in addition that
$T \supseteq I\Sigma_1$. Iterated reflection progressions are
commonly defined by
\[
  T_0 := T,
  \qquad
  T_{\alpha+1} := T_\alpha + \RFN_\Gamma(T_\alpha),
\]
where $\RFN_\Gamma(T)$ denotes the uniform reflection schema over
$\Gamma$, with limit stages obtained by unions (or, more generally,
along ordinal notation systems, cf.~\cite{Fef62, Beklemishev2005,Beklemishev2010}).
In Beklemishev's framework, such progressions admit a natural
interpretation in terms of the modal operators $[n]$ of $\GLP$, where
successive reflection principles correspond to increasing modal
strength, and consistency extensions are represented by formulas of
the form $\langle n\rangle \top$.\\\\ 
Selector-proofs belong to a closely related, layer. They provide a uniform proof-producing mechanism for local
consistency claims, and hence offer a certain constructive form of reflection. In the finite-fragment version, given a natural
number \(d\), regarded as a putative derivation code, the selector
extracts the finite fragment \(T_d\subseteq T\) actually used by \(d\)
and appeals to a \(T\)-proof of \(\Con(T_d)\). What is uniform here is
therefore not a global truth assertion, but the production of the
corresponding local proof object.\\\\  A reflection principle allows one to pass from provability
to truth, schematically or uniformly. A selector-proof keeps the proof object explicit and produces, for
each putative derivation, a proof excluding that particular derivation
as a proof of contradiction. Thus, the selector construction yields
uniform proof production without collapsing into a uniform soundness
principle.\\\\ 
This also clarifies the relation to Mostowski's Reflexivity Theorem, which supplies consistency information for finite
fragments of a theory. The selector construction makes this information
effective and proof-producing in the relevant finite data: Given a
putative derivation, it keeps track -- at least theoretically -- of the finite fragment actually
used and produces the corresponding local proof object. What it does
not provide is a general passage from provability to truth. In
particular, it does not yield a uniform reflection principle of the
form
\[
        \Prov_T(\ulcorner A\urcorner)\rightarrow A,
\]
nor does it yield the universal consistency sentence \(\Con(T)\).\\\\
Thus, selector-proofs occupy a position just below uniform reflection
principles studied in reflection progressions and
provability-logical hierarchies. This is also why they coexist with
G\"odel's 2.UVS: They give uniform control of proof objects, but not a full
internalization of consistency as a single global assertion. It is tempting to think that {\it hence} such an internalization was a red herring to begin with.



\subsection{The role of the selector $h_T$ in a (philosophical-technical) nutshell} Form the perspective of a philosopher, who is not overly interested in pure logic, it is instructive --  although from the perspective of a logician not entirely accurate -- to allow oneself to view the selector $h_T$ in a selector-proof of consistency as a primitive recursive function that maps any given model $\mathcal N$ of PA to itself, $h_T: \mathcal N\rightarrow \mathcal N$, with the property that $h_T(n)$ is a standard natural number, if $n$ is. Necessarily, also non-standard natural numbers will appear in the image of $h_T$, for non-standard models $\mathcal N$ of PA. However -- and that might be take as the underlying philosophical-technincal clue of the selector, though, again, we warn the reader that this is logically not perfectly accurate -- $h_T$ will not ``match'' a non-standard natural number, which encodes a non-standard proof of a contradiction, i.e., $h_T$ avoids G\"odel's ``monsters''.\\\\ 
From this generously vague perspective, introducing the selector $h_T$ repairs the defect that one cannot define the standard numerals, i.e., the standard model $\N$, within a non-standard model $\mathcal N$ of PA by means of first-order predicate logic: One may arrange that $h_T$ maps standard natural numbers on standard natural numbers and non-standard natural numbers on non-standard ones, while bypassing all those non-standard natural numbers, which possible encode a proof of a contradiction. I.e., instead of ``$\N$ inside $\mathcal N$'', which one might be tempted to define {\it a priori} in order to stay within the realm of honest-to-goodness proof-codes and to circumvent the non-standard natural numbers, which ``cause'' G\"odel's 2.UVS, one defines a suitably bigger realm via $h_T$, but which is small enough to exclude all those problematic non-standard numerals.


\begin{thebibliography}{99}

\bibitem[Art01]{Artemov2001}
S.~N. Artemov,
Explicit provability and constructive semantics,
\emph{Bulletin of Symbolic Logic} \textbf{7} (2001), no.~1, 1--36.

\bibitem[Art19]{Artemov2019}
S.~N. Artemov,
The provability of consistency,
arXiv:1902.07404, 2019.

\bibitem[Art25a]{Artemov2025JLC}
S.~N. Artemov,
Serial properties, selector-proofs and the provability of consistency,
\emph{Journal of Logic and Computation} \textbf{35} (2025), no.~3, exae034.

\bibitem[Art25b]{Artemov2025}
S.~N. Artemov,
Consistency formula is strictly stronger in PA than PA-consistency,
arXiv:2508.20346, 2025.

\bibitem[Aue78]{Auerbach1978}
D.~D. Auerbach,
\emph{Expressing Consistency: G\"odel's Second Incompleteness Theorem and Intensionality in Metamathematics},
Ph.D. thesis, Massachusetts Institute of Technology, 1978.

\bibitem[Aue89]{Auerbach1989}
D.~D. Auerbach,
Review of Michael Detlefsen, \emph{Hilbert's Program: An Essay on Mathematical Instrumentalism},
\emph{Journal of Symbolic Logic} \textbf{54} (1989), no.~2, 620--622.

\bibitem[Bek95]{Beklemishev1995}
L.~D. Beklemishev,
Iterated local reflection versus iterated consistency,
\emph{Annals of Pure and Applied Logic} \textbf{75} (1995), 25--48.

\bibitem[Bek05]{Beklemishev2005}
L.~D. Beklemishev,
Reflection principles and provability algebras in formal arithmetic,
\emph{Russian Mathematical Surveys} \textbf{60} (2005), no.~2, 197--268.

\bibitem[Bek10]{Beklemishev2010}
L.~D. Beklemishev,
Calibrating provability logic: from modal logic to reflection calculus,
\emph{Russian Mathematical Surveys} \textbf{65} (2010), no.~5, 857--899.

\bibitem[Det79]{Detlefsen1979}
M.~Detlefsen,
On interpreting G\"odel's second theorem,
\emph{Journal of Philosophical Logic} \textbf{8} (1979), 297--313.

\bibitem[Det86]{Detlefsen1986}
M.~Detlefsen,
\emph{Hilbert's Program: An Essay on Mathematical Instrumentalism},
Synthese Library, vol.~182, D. Reidel, Dordrecht, 1986.

\bibitem[Det01]{Detlefsen2001}
M.~Detlefsen,
What does G\"odel's second theorem say?,
\emph{Philosophia Mathematica} \textbf{9} (2001), no.~1, 37--71.

\bibitem[Fef60]{Feferman1960}
S.~Feferman,
Arithmetization of metamathematics in a general setting,
\emph{Fundamenta Mathematicae} \textbf{49} (1960), 35--92.

\bibitem[Fef62]{Fef62} 
S. Feferman, Transfinite recursive progressions of axiomatic theories, \emph{Journal of Symbolic Logic} \textbf{27} (1962), no. 3, 259--316. 

\bibitem[G\"od31]{Godel1931}
K.~G\"odel,
\"Uber formal unentscheidbare S\"atze der Principia Mathematica und verwandter Systeme I,
\emph{Monatshefte f\"ur Mathematik und Physik} \textbf{38} (1931), 173--198.

\bibitem[H\'aj-Pud93]{HP}
P.~H\'ajek and P.~Pudl\'ak,
\emph{Metamathematics of First-Order Arithmetic},
Perspectives in Mathematical Logic, Springer, Berlin, 1993.

\bibitem[Hil-Ber39]{HilbertBernays1939}
D.~Hilbert and P.~Bernays,
\emph{Grundlagen der Mathematik II},
Springer, Berlin, 1939.

\bibitem[Jer71]{Jeroslow1971}
R.~G. Jeroslow,
Consistency statements in formal arithmetic,
\emph{Fundamenta Mathematicae} \textbf{72} (1971), 17--40.

\bibitem[Lan79]{LanglandsCorvallis}
R.~P. Langlands,
Automorphic representations, Shimura varieties, and motives. Ein M\"archen,
in: \emph{Automorphic Forms, Representations and L-Functions},
Proc. Sympos. Pure Math., vol.~33, Part~2, Amer. Math. Soc., Providence, RI, 1979,
pp.~205--246.

\bibitem[L\"ob55]{Lob1955}
M.~H. L\"ob,
Solution of a problem of Leon Henkin,
\emph{Journal of Symbolic Logic} \textbf{20} (1955), no.~2, 115--118.

\bibitem[Mos52]{Mos52} 
A. Mostowski, \emph{Sentences Undecidable in Formalized Arithmetic: An Exposition of the Theory of Kurt G\"odel}, North-Holland, Amsterdam, 1952. 

\bibitem[Pud86]{Pudlak1986}
P.~Pudl\'ak,
On the length of proofs of finitistic consistency statements in first
order theories,
in J.~B. Paris, A.~J. Wilkie and G.~M. Wilmers (eds.),
\emph{Logic Colloquium '84},
North-Holland, Amsterdam, 1986, pp.~165--196.

\bibitem[Res74]{Resnik1974}
M.~D. Resnik,
On the philosophical significance of consistency proofs,
\emph{Journal of Philosophical Logic} \textbf{3} (1974), no.~1/2, 133--147.

\bibitem[Smo85]{Smorynski1985}
C.~Smory\'nski,
\emph{Self-Reference and Modal Logic},
Universitext, Springer-Verlag, New York, 1985.

\bibitem[Ste91]{Steiner1991}
M.~Steiner,
Review of Michael Detlefsen, \emph{Hilbert's Program: An Essay on Mathematical Instrumentalism},
\emph{Journal of Philosophy} \textbf{88} (1991), no.~6, 331--336.

\bibitem[Ver-Vis94]{Ver-Vis94} 
R. Verbrugge and A. Visser, A small reflection principle for bounded arithmetic, \emph{Journal of Symbolic Logic} \textbf{59} (1994), no. 3, 785--812.

\bibitem[Vis01]{Visser2001}
A.~Visser,
Provability logic,
in D.~M. Gabbay and F.~Guenthner (eds.), \emph{Handbook of Philosophical Logic}, 2nd ed., vol.~4, Kluwer, Dordrecht, 2001, 1--107.

\end{thebibliography}
\end{document}